\documentclass[reqno]{amsart}
\usepackage{amssymb}
\usepackage[colorlinks=true,hypertex]{hyperref}
\allowdisplaybreaks[4]
\theoremstyle{plain}
\newtheorem{thm}{Theorem}

\theoremstyle{remark}
\newtheorem{rem}{Remark}
\DeclareMathOperator{\td}{d\mspace{-2mu}}
\date{Completed on Wednesday 23 July 2008 in Melbourne}
\date{}

\begin{document}

\title[Conditions for a function to be completely monotonic]
{Necessary and sufficient conditions for a function involving divided differences of the di- and tri-gamma functions to be completely monotonic}

\author[F. Qi]{Feng Qi}
\address[F. Qi]{Research Institute of Mathematical Inequality Theory, Henan Polytechnic University, Jiaozuo City, Henan Province, 454010, China}
\email{\href{mailto: F. Qi <qifeng618@gmail.com>}{qifeng618@gmail.com}, \href{mailto: F. Qi <qifeng618@hotmail.com>}{qifeng618@hotmail.com}, \href{mailto: F. Qi <qifeng618@qq.com>}{qifeng618@qq.com}}
\urladdr{\url{http://qifeng618.spaces.live.com}}

\author[B.-N. Guo]{Bai-Ni Guo}
\address[B.-N. Guo]{School of Mathematics and Informatics, Henan Polytechnic University, Jiaozuo City, Henan Province, 454010, China}
\email{\href{mailto: B.-N. Guo <bai.ni.guo@gmail.com>}{bai.ni.guo@gmail.com}, \href{mailto: B.-N. Guo <bai.ni.guo@hotmail.com>}{bai.ni.guo@hotmail.com}}
\urladdr{\url{http://guobaini.spaces.live.com}}

\begin{abstract}
In the present paper, necessary and sufficient conditions are established for a function involving divided differences of the digamma and trigamma functions to be completely monotonic. Consequently, necessary and sufficient conditions are derived for a function involving the ratio of two gamma functions to be logarithmically completely monotonic, and some double inequalities are deduced for bounding divided differences of polygamma functions.
\end{abstract}

\keywords{Necessary and sufficient condition, complete monotonicity, logarithmically complete monotonicity, divided difference, trigamma function, tetragamma function, bound, polygamma function, ratio of two gamma functions}

\subjclass[2000]{Primary 26A48, 33B15; Secondary 26A51, 26D10, 65R10}

\thanks{The first author was partially supported by the China Scholarship Council}

\thanks{This paper was typeset using \AmS-\LaTeX}

\maketitle

\section{Introduction}
Recall \cite[Chapter~XIII]{mpf-1993} and \cite[Chapter~IV]{widder} that a function $f$ is said to be completely monotonic (CM) on an interval $I$ if $f$ has derivatives of all orders on $I$ and
\begin{equation}\label{CM-dfn}
(-1)^{n}f^{(n)}(x)\ge0
\end{equation}
for $x\in I$ and $n\ge0$. The well-known Bernstein-Widder's Theorem \cite[p.~160, Theorem~12a]{widder} states that a function $f(x)$ on $[0,\infty)$ is CM if and only if there exists a bounded and non-decreasing function $\alpha(t)$ such that
\begin{equation} \label{berstein-1}
f(x)=\int_0^\infty e^{-xt}\td\alpha(t)
\end{equation}
converges for $x\in[0,\infty)$. This expresses that a CM function $f$ on $[0,\infty)$ is a Laplace transform of the measure $\alpha$.
\par
Recall also~\cite{Atanassov, minus-one} that a function $f$ is said to be logarithmically completely monotonic (LCM) on an interval $I\subseteq\mathbb{R}$ if it has derivatives of all orders on $I$ and its logarithm $\ln f$ satisfies
\begin{equation}\label{lcm-dfn}
(-1)^k[\ln f(x)]^{(k)}\ge0
\end{equation}
for $k\in\mathbb{N}$ on $I$.
The terminology ``logarithmically completely monotonic function'' was first put forward in~\cite{Atanassov} without an explicit definition, but it seems to have been ignored until recently by the mathematical community. In early 2004, this notion was recovered in~\cite{minus-one, auscm-rgmia}. Since the class of LCM functions is a subclass of the CM functions, this definition is significant and meaningful. For more information on basic properties of LCM functions, please refer to \cite{CBerg, grin-ismail, e-gam-rat-comp-mon} and related references therein.
\par
It is well-known that the classical Euler's gamma function
\begin{equation}\label{gamma-dfn}
\Gamma(x)=\int^\infty_0t^{x-1} e^{-t}\td t
\end{equation}
for $x>0$, the psi function $\psi(x)=\frac{\Gamma'(x)}{\Gamma(x)}$ and the polygamma functions $\psi^{(i)}(x)$ for $i\in\mathbb{N}$ are a series of important special functions and have much extensive applications in many branches such as statistics, probability, number theory, theory of $0$-$1$ matrices, graph theory, combinatorics, physics, engineering, and other mathematical sciences. In particular, the functions $\psi(x)$ and $\psi'(x)$ for $x>0$ are also called the digamma and trigamma functions respectively, see \cite{abram} and \cite[p.~71]{dict-bullen}.
\par
By using the double inequalities
\begin{equation}
\frac1x+\frac1{2x^2}+\frac1{6x^3}-\frac1{30x^5}<\psi'(x)<\frac1x+\frac1{2x^2}+\frac1{6x^3},
\end{equation}
see \cite[p.~860, Theorem~4]{gordon}, and
\begin{equation}
-\frac1{x^2}-\frac1{x^3}-\frac1{2x^4}<\psi''(x)<-\frac1{x^2}-\frac1{x^3},
\end{equation}
a special cases of \cite[Theorem~9]{alzer-mc-97}, for $x>0$, in order to show that the double inequality
\begin{equation}
(n-1)!\exp\biggl[\frac{\alpha}x-n\psi(x)\biggr]<\bigl\vert\psi^{(n)}(x)\bigr\vert  <(n-1)!\exp\biggl[\frac{\beta}x-n\psi(x)\biggr]
\end{equation}
holds for $x>0$ if and only if $\alpha\le-n$ and $\beta\ge0$, it was established in the proof of  \cite[Theorem~4.8]{forum-alzer} that
\begin{equation}\label{psi'psi''}
[\psi'(x)]^2+\psi''(x)>\frac{p(x)}{900x^4(x+1)^{10}}
\end{equation}
for $x>0$, where
\begin{equation}
\begin{split}
p(x)&=75x^{10}+900x^9+4840x^8+15370x^7+31865x^6+45050x^5\\
&\quad+44101x^4+29700x^3+13290x^2+3600x+450.
\end{split}
\end{equation}
\par
From \eqref{psi'psi''}, the inequality
\begin{equation}\label{positivity}
[\psi'(x)]^2+\psi''(x)>0
\end{equation}
for $x>0$ was deduced and used to present a double inequality
\begin{equation}\label{batir-ineq-interesting}
\exp\bigl\{\alpha\bigl[e^{\psi(x)}\psi(x)-e^{\psi(x)}+1\bigr]\bigr\}\le \frac{\Gamma(x)}{\Gamma(c)} \le\exp\bigl\{\beta\bigl[e^{\psi(x)}\psi(x)-e^{\psi(x)}+1\bigr]\bigr\}
\end{equation}
in the proof of~\cite[Theorem~2.1]{batir-interest-jipam} and~\cite[Theorem~2.1]{batir-interest-rgmia}, with $\alpha=1$ and $\beta=\frac{6e^\gamma}{\pi^2}$ for $x>c$, where $c=1.4616\dotsm$ is the only positive zero of the psi function $\psi(x)$ on $(0,\infty)$.
\par
In~\cite[Theorem~2.1]{batir-new-jipam} and~\cite[Theorem~2.1]{batir-new-rgmia}, in order to prove that the inequality
\begin{equation}\label{batir-one-side-ineq}
\psi(x)>\ln\frac{\pi^2}6-\gamma-\ln\bigl(e^{1/x}-1\bigr)
\end{equation}
holds for $x\ge2$, the inequality~\eqref{positivity} was recovered in~\cite[Lemma~1.1]{batir-new-jipam} and~\cite[Lemma~1.1]{batir-new-rgmia} elegantly.
\par
In \cite{property-psi-ii.tex}, the inequality~\eqref{positivity} was used to give a simple proof for the increasing property of the function
\begin{equation}\label{phi(x)-dfn}
\phi(x)=\psi(x)+\ln\bigl(e^{1/x}-1\bigr)
\end{equation}
on $(0,\infty)$.
\par
In \cite[Remark~1.3]{batir-jmaa-06-05-065}, it was pointed out that the inequality~\eqref{positivity} is a special case of the inequality
\begin{equation}
(-1)^n\psi^{(n+1)}(x)<\frac{n}{\sqrt[n]{(n-1)!}\,}\bigl[(-1)^{n-1}\psi^{(n)}(x)\bigr]^{1+1/n}
\end{equation}
for $x>0$ and $n\in\mathbb{N}$.
\par
In \cite[Theorem~4.3]{alzer-grinshpan}, the inequality~\eqref{positivity} was applied to provide a sharp and generalized version of \eqref{batir-ineq-interesting}:\label{intr-sec} For $0<a<b\le\infty$ and $x\in(a,b)$, the inequality~\eqref{batir-ineq-interesting} is valid with the best possible constant factors
\begin{equation}
\alpha=\begin{cases}
Q(b), &\text{if $b<\infty$}\\ 1,&\text{if $b=\infty$}
\end{cases}
\quad\text{and}\quad \beta=Q(a),
\end{equation}
where
\begin{equation}
Q(x)=\begin{cases}
\dfrac{\ln\Gamma(x)-\ln\Gamma(c)}{[\psi(x)-1]e^{\psi(x)}+1},&x\ne c;\\[1em]
\dfrac1{\psi'(c)},& x=c.
\end{cases}
\end{equation}
\par
In \cite[Lemma~4.6]{alzer-grinshpan} and \cite[Theorem~4.8]{alzer-grinshpan}, the inequalities \eqref{positivity} and~\eqref{batir-ineq-interesting} were respectively generalized to $q$-analogues.
\par
In \cite[Theorem~2]{Infinite-family-Digamma.tex}, the inequality~\eqref{positivity} was used to show that the function $e^{\psi(x+1)}-x$ is strictly decreasing and strictly convex on $(-1,\infty)$.
\par
In \cite{AAM-Qi-09-PolyGamma.tex}, among other things, it was proved that the function
\begin{equation}\label{di-tetra-gamma-lambda}
\Delta_{\lambda}(x)=[\psi'(x)]^2+\lambda\psi''(x)
\end{equation}
is CM on $(0,\infty)$ if and only if $\lambda\le1$.
\par
In~\cite[Theorem~1]{egp}, it was proved that the function
\begin{equation}\label{z-s-t-(x)}
z_{s,t}(x)=\begin{cases}
\bigg[\dfrac{\Gamma(x+t)}{\Gamma(x+s)}\bigg]^{1/(t-s)}-x,&s\ne t\\
e^{\psi(x+s)}-x,&s=t
\end{cases}
\end{equation}
on $(-\alpha,\infty)$ for real numbers $s$ and $t$ and $\alpha=\min\{s,t\}$ is either convex and decreasing for $|t-s|<1$ or concave and increasing for $|t-s|>1$. In order to provide an alternative proof for \cite[Theorem~1]{egp}, the function
\begin{equation}\label{Delta-dfn}
\Delta_{s,t}(x)=\begin{cases}\bigg[\dfrac{\psi(x+t) -\psi(x+s)}{t-s}\bigg]^2
+\dfrac{\psi'(x+t)-\psi'(x+s)}{t-s},&s\ne t\\
[\psi'(x+s)]^2+\psi''(x+s),&s=t
\end{cases}
\end{equation}
for $|t-s|<1$ and $-\Delta_{s,t}(x)$ for $|t-s|>1$ are proved in \cite{notes-best-simple-open.tex, notes-best-simple.tex} to be CM on $(-\alpha,\infty)$.
\par
Using the complete monotonicity of the function~\eqref{Delta-dfn}, the inequality \eqref{batir-ineq-interesting} and \cite[Theorem~4.3]{alzer-grinshpan} mentioned on page~\pageref{intr-sec} were generalized in \cite[Theorem~5]{notes-best-simple-open.tex-rev} to a monotonic property as follows: For real numbers $s$ and $t$, $\alpha=\min\{s,t\}$ and $c\in(-\alpha,\infty)$, let
\begin{equation}
g_{s,t}(x)=\begin{cases}\displaystyle
\frac1{t-s} \int_c^x\ln\biggl[\frac{\Gamma(u+t)}{\Gamma(u+s)} \frac{\Gamma(c+s)}{\Gamma(c+t)}\biggr]\td u,&s\ne t\\[1em]
\displaystyle
\int_c^x[\psi(u+s)-\psi(c+s)]\td u,&s=t
\end{cases}
\end{equation}
on $x\in(-\alpha,\infty)$. Then the function
\begin{equation}
f_{s,t}(x)=\begin{cases}\displaystyle
\frac{g_{s,t}(x)}{[g'_{s,t}(x)-1]\exp[g'_{s,t}(x)]+1},&x\ne c\\[1em]
\dfrac1{g''_{s,t}(c)},&x=c
\end{cases}
\end{equation}
on $(-\alpha,\infty)$ is decreasing for $|s-t|<1$ and increasing for $|s-t|>1$.
\par
In \cite{notes-best-simple-open.tex-rev, notes-best-simple-rev.tex}, some other applications of the complete monotonicity of the function~\eqref{Delta-dfn} were also demonstrated.
\par
For real numbers $s$, $t$, $\alpha=\min\{s,t\}$ and $\lambda$, define
\begin{equation}\label{Delta-lambda-dfn}
\Delta_{s,t;\lambda}(x)=\begin{cases}\bigg[\dfrac{\psi(x+t) -\psi(x+s)}{t-s}\bigg]^2
+\lambda\dfrac{\psi'(x+t)-\psi'(x+s)}{t-s},&s\ne t\\
[\psi'(x+s)]^2+\lambda\psi''(x+s),&s=t
\end{cases}
\end{equation}
on $(-\alpha,\infty)$. It is clear that $\Delta_{s,t;\lambda}(x)=\Delta_{s,t}(x)$ and $\Delta_{s,t;\lambda}(x)=\Delta_{t,s;\lambda}(x)$.
\par
The aim of this paper is to present necessary and sufficient conditions for the function $\Delta_{s,t;\lambda}(x)$ to be CM on $(-\alpha,\infty)$.
\par
Our main results can be stated as the following Theorem~\ref{CMDT-divided-thm}.

\begin{thm}\label{CMDT-divided-thm}
The function $\Delta_{s,t;\lambda}(x)$ has the following CM properties:
\begin{enumerate}
\item
For $0<|t-s|<1$,
\begin{enumerate}
\item
the function $\Delta_{s,t;\lambda}(x)$ is CM on $(-\alpha,\infty)$ if and only if $\lambda\le1$,
\item
so is the function $-\Delta_{s,t;\lambda}(x)$ if and only if $\lambda\ge\frac1{|t-s|}$;
\end{enumerate}
\item
For $|t-s|>1$,
\begin{enumerate}
\item
the function $\Delta_{s,t;\lambda}(x)$ is CM on $(-\alpha,\infty)$ if and only if $\lambda\le\frac1{|t-s|}$,
\item
so is the function $-\Delta_{s,t;\lambda}(x)$ if and only if $\lambda\ge1$;
\end{enumerate}
\item
For $s=t$, the function $\Delta_{s,s;\lambda}(x)$ is CM on $(-s,\infty)$ if and only if $\lambda\le1$;
\item
For $|t-s|=1$,
\begin{enumerate}
\item
the function $\Delta_{s,t;\lambda}(x)$ is CM if and only if $\lambda<1$,
\item
so is the function $-\Delta_{s,t;\lambda}(x)$ if and only if $\lambda>1$,
\item
and $\Delta_{s,t;1}(x)\equiv0$.
\end{enumerate}
\end{enumerate}
\end{thm}

As a consequence of Theorem~\ref{CMDT-divided-thm}, the following logarithmically complete monotonicity of a function involving the ratio of two gamma functions is deduced.

\begin{thm}\label{CMDT-divided-thm-2}
For real numbers $s$, $t$, $\alpha=\min\{s,t\}$ and $\lambda$, let
\begin{equation}\label{Delta-gamma-dfn}
\mathcal{H}_{s,t;\lambda}(x)=
\begin{cases}
\dfrac{(x+t)^{[1/(t-s)-\lambda]/2}}{(x+s)^{[1/(t-s)+\lambda]/2}} \bigg[\dfrac{\Gamma(x+t)}{\Gamma(x+s)}\bigg]^{1/(t-s)},&s\ne t\\[0.6em]
\dfrac1{(x+s)^\lambda}\exp\biggl[\psi(x+s)+\dfrac1{2(x+s)}\biggr],&s=t
\end{cases}
\end{equation}
on $(-\alpha,\infty)$.
\begin{enumerate}
  \item
For $0<|t-s|<1$,
\begin{enumerate}
\item
the function $\mathcal{H}_{s,t;\lambda}(x)$ is LCM on $(-\alpha,\infty)$ if and only if $\lambda\ge\frac1{|t-s|}$,
\item
so is the function $[\mathcal{H}_{s,t;\lambda}(x)]^{-1}$ if and only if $\lambda\le1$;
\end{enumerate}
  \item
For $|t-s|>1$,
\begin{enumerate}
\item
the function $\mathcal{H}_{s,t;\lambda}(x)$ is LCM on $(-\alpha,\infty)$ if and only if $\lambda\ge1$,
\item
so is the function $[\mathcal{H}_{s,t;\lambda}(x)]^{-1}$ if and only if $\lambda\le\frac1{|t-s|}$;
\end{enumerate}
  \item
For $s=t$, the function $[\mathcal{H}_{s,s;\lambda}(x)]^{-1}$ is LCM on $(-\alpha,\infty)$ if and only if $\lambda\le1$;
  \item
For $|t-s|=1$,
\begin{enumerate}
\item
the function $\mathcal{H}_{s,t;\lambda}(x)$ is LCM on $(-\alpha,\infty)$ if and only if $\lambda>1$,
\item
so is the function $[\mathcal{H}_{s,s;\lambda}(x)]^{-1}$ if and only if $\lambda<1$,
\item
and $\mathcal{H}_{s,t;1}(x)\equiv1$.
\end{enumerate}
\end{enumerate}
\end{thm}

As consequences of Theorem~\ref{CMDT-divided-thm-2}, some inequalities for the ratio of two gamma functions and divided differences of polygamma functions are derived as follows.

\begin{thm}\label{CMDT-divided-thm-3}
For positive numbers $a$ and $b$, the inequality
\begin{equation}\label{gamma-ratio-sqrt-ineq}
\bigg[\frac{\Gamma(b)}{\Gamma(a)}\bigg]^{1/(b-a)} < \sqrt{ab}\,\biggl(\frac{a}b\biggr)^{1/2(b-a)}
\end{equation}
holds for $0<|b-a|<1$ and reverses for $|b-a|>1$; For $0<|b-a|<1$, the double inequality
\begin{multline}\label{final-thm3-square-ab}
\frac{(k-1)!}2\biggl[\biggl(\frac1{b-a}+\beta\biggr)\frac1{a^k} +\biggl(\beta-\frac1{b-a}\biggr)\frac1{b^k}\biggr]\\* <\frac{(-1)^{k-1}\bigl[\psi^{(k-1)}(b)-\psi^{(k-1)}(a)\bigr]}{b-a}\\* <\frac{(k-1)!}2\biggl[\biggl(\frac1{b-a}+\gamma\biggr)\frac1{a^k} +\biggl(\gamma-\frac1{b-a}\biggr)\frac1{b^k}\biggr]
\end{multline}
on $(-\alpha,\infty)$ holds if and only if $\beta\le1$ and $\gamma\ge\frac1{|b-a|}$; For $|b-a|>1$, the inequalities in~\eqref{final-thm3-square-ab} are valid if and only if $\beta\le\frac1{|b-a|}$ and $\gamma\ge1$.
\end{thm}

After proving the above theorems in next section, we would also like to give some remarks on the above theorems and related results to the inequality~\eqref{batir-one-side-ineq} and the increasing property of $\phi(x)$ defined by~\eqref{phi(x)-dfn} in the final section of this paper.

\section{Proofs of theorems}

Now we are in a position to prove the above theorems.

\begin{proof}[Proof of Theorem~\ref{CMDT-divided-thm}]
For $s=t$, it has been proved in \cite{AAM-Qi-09-PolyGamma.tex} that the function $\Delta_{s,s;\lambda}(x)$ is CM on $(-s,\infty)$ if and only if $\lambda\le1$.
\par
It is easy to calculate by integrating in part in \eqref{gamma-dfn} that
\begin{equation}\label{gamma-recurrence}
\Gamma(x+1)=x\Gamma(x), \quad x>0.
\end{equation}
Taking the logarithm of equation~\eqref{gamma-recurrence} and differentiating $k\in\mathbb{N}$ times consecutively on both sides give
\begin{equation}\label{psisymp4}
\psi^{(k-1)}(x+1)=\psi^{(k-1)}(x)+(-1)^{k-1}\frac{(k-1)!}{x^{k}},\quad x>0,\quad k\in\mathbb{N}.
\end{equation}
\par
For $s-t=\pm1$, using \eqref{psisymp4} gives
\begin{equation}
\Delta_{s,s\mp1;\lambda}(x)=
\begin{cases}
\dfrac{1-\lambda}{(x+s-1)^2};\\[0.8em]
\dfrac{1-\lambda}{(x+s)^2}.
\end{cases}
\end{equation}
As a result, the function $\Delta_{s,s\mp1;\lambda}(x)$ is CM if and only if $\lambda<1$, so is the function $-\Delta_{s,s\mp1;\lambda}(x)$ if and only if $\lambda>1$, and $\Delta_{s,s\mp1;1}(x)\equiv0$.
\par
For $0<|s-t|\ne1$, direct calculation and utilization of \eqref{psisymp4} yield
\begin{multline*}
\Delta_{s,t;\lambda}(x)-\Delta_{s,t;\lambda}(x+1)=\frac1{(t-s)^2}\bigl\{[\psi(x+t)-\psi(x+t+1)]\\
\begin{aligned}
&+[\psi(x+s+1)-\psi(x+s)]\bigr\}\bigl\{[\psi(x+t)+\psi(x+t+1)]\\
&-[\psi(x+s)+\psi(x+s+1)]\bigr\}+\frac{\lambda}{t-s}\bigl\{[\psi'(x+t)-\psi'(x+t+1)]\\ &-[\psi'(x+s)-\psi'(x+s+1)]\bigr\}
\end{aligned}\\
\begin{aligned}
&=\frac1{(t-s)^2}\biggl(\frac1{x+s}-\frac1{x+t}\biggr)\biggl\{2[\psi(x+t)-\psi(x+s)] +\biggl(\frac1{x+s}-\frac1{x+t}\biggr)\biggr\}\\
&\quad+\frac{\lambda}{t-s}\biggl[\frac1{(x+t)^2}-\frac1{(x+s)^2}\biggr]
\end{aligned}\\
\begin{aligned}
&=\frac2{(x+s)(x+t)}\biggl[\frac{\psi(x+t)-\psi(x+s)}{t-s}-\frac1{2(x+s)(x+t)} -\frac{\lambda(2x+s+t)}{2(x+s)(x+t)}\biggr]\\
&\triangleq\frac{2\theta_{s,t;\lambda}(x)}{(x+s)(x+t)}.
\end{aligned}
\end{multline*}\label{theta-lambda-dfn}
\par
From \eqref{gamma-dfn}, it is easy to deduce that
\begin{equation}\label{fracint}
\frac1{x^\omega}=\frac1{\Gamma(\omega)}\int_0^\infty t^{\omega-1}e^{-xt}\td t
\end{equation}
for real numbers $x>0$ and $\omega>0$. For $x>0$, it was listed in \cite[p.~259, 6.3.22]{abram} that
\begin{equation}\label{psi}
\psi(x)=-\gamma+\int_0^\infty\frac{e^{-t}-e^{-xt}}{1-e^{-t}}\td t.
\end{equation}
By virtue of formulas \eqref{fracint} and \eqref{psi}, it follows that
\begin{align*}
\theta_{s,t;\lambda}(x)&=\frac{\psi(x+t)-\psi(x+s)}{t-s} -\frac12\biggl[\biggl(\frac1{t-s}+\lambda\biggr)\frac1{x+s} +\biggl(\lambda-\frac1{t-s}\biggr)\frac1{x+t}\biggr]\\
&=\int_0^\infty\frac{e^{-su}-e^{-tu}}{(t-s)(1-e^{-u})}e^{-xu}\td u -\frac1{2(t-s)}\int_0^\infty\bigl(e^{-su}-e^{-tu}\bigr)e^{-xu}\td u \\ &\quad-\frac\alpha2\int_0^\infty\bigl(e^{-su}+e^{-tu}\bigr)e^{-xu}\td u\\
&=\int_0^\infty\biggl[\frac1{t-s}\biggl(\frac1{1-e^{-u}}-\frac12\biggr) \frac{e^{-su}-e^{-tu}}{e^{-su}+e^{-tu}}-\frac{\lambda}2\biggr] \bigl(e^{-su}+e^{-tu}\bigr)e^{-xu}\td u\\
&=\frac12\int_0^\infty\biggl[\frac{\tanh((t-s)u/2)}{(t-s)\tanh(u/2)} -\lambda\biggr]\bigl(e^{-su}+e^{-tu}\bigr)e^{-xu}\td u.
\end{align*}
It is not difficult to obtain that
\begin{equation}
\lim_{t\to\infty}\frac{\tanh((t-s)u/2)}{(t-s)\tanh(u/2)}=\frac1{|t-s|} \quad\text{and}\quad \lim_{t\to0^+}\frac{\tanh((t-s)u/2)}{(t-s)\tanh(u/2)}=1.
\end{equation}
Straightforward differentiation gives
\begin{equation*}
\frac{\td}{\td u}\biggl[\frac{\tanh((t-s)u/2)}{(t-s)\tanh(u/2)}\biggr] =\frac{u}{4\sinh^2(u/2)\cosh^2((t-s)u/2)}\biggl[\frac{\sinh u}u -\frac{\sinh((t-s)u)}{(t-s)u}\biggr].
\end{equation*}
Since the function $\frac{\sinh u}u$ is even and increasing on $(-\infty,\infty)\setminus\{0\}$, then
\begin{equation*}
\frac{\td}{\td u}\biggl[\frac{\tanh((t-s)u/2)}{(t-s)\tanh(u/2)}\biggr]
\begin{cases}
<0,&\text{if $|t-s|>1$;} \\
>0,& \text{if $0<|t-s|<1$.}
\end{cases}
\end{equation*}
Hence, the function $\frac{\tanh((t-s)u/2)}{(t-s)\tanh(u/2)}$ on $(0,\infty)$ is increasing for $0<|t-s|<1$ and decreasing for $|t-s|>1$. Therefore,
\begin{enumerate}
\item
the function $\theta_{s,t;\lambda}(x)$ is CM on $(-\alpha,\infty)$ if
\begin{enumerate}\label{i-ii-label}
\item[(i)]
either $0<|t-s|<1$ and $\lambda\le1$
\item[(ii)]
or $|t-s|>1$ and $\lambda\le\frac1{|t-s|}$;
\end{enumerate}
\item
the function $-\theta_{s,t;\lambda}(x)$ is CM on $(-\alpha,\infty)$ if
\begin{enumerate}\label{iii-vi-label}
\item[(iii)]
either $0<|t-s|<1$ and $\lambda\ge\frac1{|t-s|}$
\item[(iv)]
or $|t-s|>1$ and $\lambda\ge1$.
\end{enumerate}
\end{enumerate}
\par
Further, since the product of any finite CM functions is still CM and the function $\frac2{(x+s)(x+t)}$ is CM, then the function $\Delta_{s,t;\lambda}(x)-\Delta_{s,t;\lambda}(x+1)$ is CM on $(-\alpha,\infty)$ under the above condition (i) or (ii) and the function $\Delta_{s,t;\lambda}(x+1)-\Delta_{s,t;\lambda}(x)$ is CM on $(-\alpha,\infty)$ under the condition (iii) or (iv) above.
\par
For $x>0$ and $k\in\mathbb{N}$, it was listed in \cite[p.~260, 6.4.1]{abram} that
\begin{equation}
\psi ^{(k)}(x)=(-1)^{k+1}\int_{0}^{\infty}\frac{t^{k}}{1-e^{-t}}e^{-xt}\td t. \label{psim}
\end{equation}
This implies $\lim_{x\to\infty}\psi ^{(k)}(x)=0$, and so
\begin{align*}
\Delta_{s,t;\lambda}^{(k-1)}(x)&=\biggl\{\biggl[\frac1{t-s}\int_s^t\psi'(x+u)\td u\biggr]^2\biggr\}^{(k-1)}+\lambda\dfrac{\psi^{(k)}(x+t)-\psi^{(k)}(x+s)}{t-s}\\
&=\frac1{t-s}\Biggl[\sum_{i=0}^{k-1}\binom{k-1}{i}\int_s^t\psi^{(i+1)}(x+u)\td u \int_s^t\psi^{k-i}(x+u)\td u\Biggr]\\
&\quad+\lambda\dfrac{\psi^{(k)}(x+t)-\psi^{(k)}(x+s)}{t-s}\\
&\to0
\end{align*}
for $k\in\mathbb{N}$ as $x\to\infty$.
\par
If $\Delta_{s,t;\lambda}(x)-\Delta_{s,t;\lambda}(x+1)$ is CM, then
\begin{align*}
&\quad(-1)^{k-1}[\Delta_{s,t;\lambda}(x)-\Delta_{s,t;\lambda}(x+1)]^{(k-1)} \\* &=(-1)^{k-1}\Delta_{s,t;\lambda}^{(k-1)}(x) -(-1)^{k-1}\Delta_{s,t;\lambda}^{(k-1)}(x+1)\\
&\ge0
\end{align*}
for $k\in\mathbb{N}$ and $x\in(-\alpha,\infty)$. Thus, in virtue of the mathematical induction and the verified fact that $\lim_{x\to\infty}\Delta_{s,t;\lambda}^{(k-1)}(x)=0$ for $k\in\mathbb{N}$, it follows that
\begin{gather*}
(-1)^{k-1}\Delta_{s,t;\lambda}^{(k-1)}(x)\ge(-1)^{k-1}\Delta_{s,t;\lambda}^{(k-1)}(x+1) \ge(-1)^{k-1}\Delta_{s,t;\lambda}^{(k-1)}(x+2)\\
(-1)^{k-1}\Delta_{s,t;\lambda}^{(k-1)}(x+3)\ge\dotsm \ge (-1)^{k-1}\Delta_{s,t;\lambda}^{(k-1)}(x+m) \to0
\end{gather*}
as $m\to\infty$. This means that the function $\Delta_{s,t;\lambda}(x)$ is CM on $(-\alpha,\infty)$.
\par
Similarly, if $\Delta_{s,t;\lambda}(x+1)-\Delta_{s,t;\lambda}(x)$ is CM on $(-\alpha,\infty)$, then the function $-\Delta_{s,t;\lambda}(x)$ is also CM on $(-\alpha,\infty)$.
\par
\par
As a consequence of either \cite[Theorem~2]{chen-qi-log-jmaa}, \cite[Theorem~2.1]{Ismail-Lorch-Muldoon}, \cite[Thorem~1.3]{sandor-gamma-2-ITSF.tex} or~\cite[Thorem~3]{sandor-gamma-2-ITSF.tex-rgmia}, the double inequality
\begin{equation}\label{qi-psi-ineq}
\frac{(k-1)!}{x^k}+\frac{k!}{2x^{k+1}}<
(-1)^{k+1}\psi^{(k)}(x)<\frac{(k-1)!}{x^k}+\frac{k!}{x^{k+1}}
\end{equation}
for $k\in\mathbb{N}$ on $(0,\infty)$ can be derived easily. This implies that
\begin{equation}\label{lim-k-infty}
\lim_{x\to\infty}\bigl[(-1)^{k+1}x^k\psi^{(k)}(x)\bigr]=(k-1)!.
\end{equation}
If the function $\Delta_{s,t;\lambda}(x)$ is CM on $(-\alpha,\infty)$, then $\Delta_{s,t;\lambda}(x)\ge0$ on  $(-\alpha,\infty)$, which is equivalent to
\begin{equation}
\lambda\le-\frac{[\psi(x+t) -\psi(x+s)]^2}{(t-s)[\psi'(x+t)-\psi'(x+s)]} =-\frac{[\psi'(x+\xi)]^2}{\psi''(x+\xi)} \to1
\end{equation}
as $x\to\infty$ by making use of the mean value theorem for derivative and \eqref{lim-k-infty}, where $\xi$ is between $s$ and $t$. On the other hand, since $\lim_{x\to0^+}\bigl[(-1)^{k}\psi^{(k-1)}(x)\bigr]=\infty$ for $k\in\mathbb{N}$, L'H\^ospital's rule gives
\begin{gather*}
\lim_{x\to(-\alpha)^+}\frac{[\psi(x+t) -\psi(x+s)]^2}{(t-s)[\psi'(x+t)-\psi'(x+s)]} =\lim_{x\to(-s)^+}\frac{[\psi(x+t) -\psi(x+s)]^2}{(t-s)[\psi'(x+t)-\psi'(x+s)]}\\
\begin{aligned}
&=\frac2{t-s}\lim_{x\to(-s)^+}\biggl[\frac{\psi(x+t)\psi'(x+s)}{\psi''(x+s)} -\frac{\psi(x+s)\psi'(x+s)}{\psi''(x+s)}\biggr] \\ &=\frac2{t-s}\lim_{u\to0^+}\biggl[\frac{\psi(u+t-s)\psi'(u)}{\psi''(u)} -\frac{\psi(u)\psi'(u)}{\psi''(u)}\biggr],
\end{aligned}
\end{gather*}
where $t>s$ and $t-s\ne1$ are assumed without loss of generality. From
\begin{equation}\label{u-psi-2-0}
\lim_{u\to0^+}[u\psi(u)]=\lim_{u\to0^+}[u\psi(u+1)-1]=-1
\end{equation}
by \eqref{psisymp4} for $k=1$ and the fact obtained in \cite[p.~182, Lemma~2.1]{forum-alzer} that the function $\frac{x\psi^{(k+1)(x)}}{\psi^{(k)}(x)}$ for $k\in\mathbb{N}$ is strictly increasing from $[0,\infty)$ onto $[-(k+1),-k)$, it is deduced readily that
\begin{equation}
\lim_{u\to0^+}\frac{\psi'(u)}{\psi''(u)}=\lim_{u\to0^+}u\lim_{u\to0^+}\frac{\psi'(u)}{u\psi''(u)}=0
\end{equation}
 and
\begin{equation}
\lim_{u\to0^+}\frac{\psi(u)\psi'(u)}{\psi''(u)}=\lim_{u\to0^+}[u\psi(u)] \lim_{u\to0^+}\frac{\psi'(u)}{u\psi''(u)} =\frac12.
\end{equation}
So
\begin{equation}
\lim_{x\to(-\alpha)^+}\frac{[\psi(x+t) -\psi(x+s)]^2}{(t-s)[\psi'(x+t)-\psi'(x+s)]}=-\frac1{|t-s|}.
\end{equation}
As a result, the necessities for the function $\Delta_{s,t;\lambda}(x)$ to be CM on $(-\alpha,\infty)$ is proved. The left proofs are similar and so omitted. The proof of Theorem~\ref{CMDT-divided-thm} is complete.
\end{proof}

\begin{proof}[Proof of Theorem~\ref{CMDT-divided-thm-2}]
For the case $t-s=1$, it is easy to see that
\begin{equation}
\mathcal{H}_{s,s+1;\lambda}(x)=[(x+s)(x+s+1)]^{(1-\lambda)/2}.
\end{equation}
Therefore, the logarithmically complete monotonicity of the function $\mathcal{H}_{s,s+1;\lambda}(x)=\mathcal{H}_{s+1,s;\lambda}(x)$ is proved.
\par
For $s=t$, it is equivalent to showing the logarithmically complete monotonicity of the function
\begin{equation}
\frac1{x^\lambda}\exp\biggl[\psi(x)+\frac1{2x}\biggr]
\end{equation}
on $(0,\infty)$, which is a direct consequence of the complete monotonicity of the function
\begin{equation}\label{psi-1-der-frac}
\psi'(x)-\frac1{2x^2}-\frac\lambda{x}
\end{equation}
on $(0,\infty)$, whose sufficiency has been verified in the proof of In order to show the necessity, it is sufficient to deduce $\lambda\le1$ from the positivity of the function $\lambda\le1$ from the positivity of the function \eqref{psi-1-der-frac}, which is equivalent to
\begin{equation}
\lambda\le x\biggl[\psi'(x)-\frac1{2x^2}\biggr]=x\psi'(x)-\frac1{2x}\to1
\end{equation}
as $x\to\infty$ by making use of \eqref{lim-k-infty}.
\par
For $0<|t-s|\ne1$, taking the logarithm of $\mathcal{H}_{s,t;\lambda}(x)$ and differentiating yields $[\ln\mathcal{H}_{s,t;\lambda}(x)]'=\theta_{s,t;\lambda}(x)$, where $\theta_{s,t;\lambda}(x)$ is the function defined in the proof of Theorem~\ref{CMDT-divided-thm} on page~\pageref{theta-lambda-dfn}. Hence, the function $\mathcal{H}_{s,t;\lambda}(x)$ is LCM on $(-\alpha,\infty)$ when the condition (iii) or (vi) on page~\pageref{iii-vi-label} is satisfied and so is the function $[\mathcal{H}_{s,t;\lambda}(x)]^{-1}$ when the condition (i) or (ii) on page~\pageref{i-ii-label} in the proof of Theorem~\ref{CMDT-divided-thm} is satisfied.
\par
If the function $\mathcal{H}_{s,t;\lambda}(x)$ is LCM on $(-\alpha,\infty)$, then $\theta_{s,t;\lambda}(x)\le0$ on $(-\alpha,\infty)$, which can be rewritten as
\begin{align*}
\lambda&\ge\frac{2(x+s)(x+t)}{2x+s+t}\biggl[\frac{\psi(x+t)-\psi(x+s)}{t-s}-\frac1{2(x+s)(x+t)}\biggr] \triangleq\Lambda_{s,t}(x).
\end{align*}
It is easy to see that
\begin{align*}
\lim_{x\to\infty}\Lambda_{s,t}(x)&=\lim_{x\to\infty} \frac{2(x+s)(x+t)[\psi(x+t)-\psi(x+s)]} {(t-s)(2x+s+t)}-\lim_{x\to\infty}\frac1{2x+s+t}\\
&=\lim_{x\to\infty} \frac{2(x+s)(x+t)}{(t-s)(2x+s+t)} \lim_{x\to\infty}\{x[\psi(x+t)-\psi(x+s)]\}\\
&=\frac1{t-s}\lim_{x\to\infty}\biggl[x\int_s^t\psi'(x+u)\td u\biggr]\\
&=\frac1{t-s}\int_s^t\lim_{x\to\infty}[x\psi'(x+u)]\td u\\
&=1
\end{align*}
by~\eqref{lim-k-infty} for $k=1$. On the other hand, if assume $t>s$ with out loss of generality, then
\begin{align*}
\lim_{x\to(-\alpha)^+}\Lambda_{s,t}(x)&=\lim_{x\to(-s)^+}\Lambda_{s,t}(x)\\
&=\lim_{u\to0^+}\frac{2u(u+t-s)}{2u+t-s}\biggl[\frac{\psi(u+t-s)-\psi(u)}{t-s} -\frac1{2u(u+t-s)}\biggr]\\
&=\lim_{u\to0^+}\frac{2u[\psi(u+t-s)-\psi(u)]}{t-s}-\frac1{t-s}\\
&=-\lim_{u\to0^+}\frac{2u\psi(u)}{t-s}-\frac1{t-s}\\
&=\frac1{t-s}
\end{align*}
by virtue of~\eqref{u-psi-2-0}. Consequently, the necessities for the function $\mathcal{H}_{s,t;\lambda}(x)$ to be LCM on $(-\alpha,\infty)$ are verified.
\par
The left proofs are similar and so omitted. Theorem~\ref{CMDT-divided-thm-2} is proved.
\end{proof}

\begin{proof}[Proof of Theorem~\ref{CMDT-divided-thm-3}]
In \cite{wendel}, the following asymptotic relation was obtained:
\begin{equation}\label{wendel-approx}
\lim_{x\to\infty}\frac{\Gamma(x+s)}{x^s\Gamma(x)}=1
\end{equation}
for real $s$ and $x$ holds. This implies that
\begin{equation}\label{infinity-limit-ratio-gamma}
\begin{split}
\mathcal{H}_{s,t;\lambda}(x) &=\frac{x+s}{[(x+s)(x+t)]^{\lambda/2}} \biggl[\dfrac{\Gamma(x+t)}{(x+s)^{t-s}\Gamma(x+s)}\biggr]^{1/(t-s)} \frac{(x+t)^{1/2(t-s)}}{(x+s)^{1/2(t-s)}}\\
&\to\begin{cases}
1,&\lambda=1\\
0,&\lambda>1\\
\infty,&\lambda<1
\end{cases}
\end{split}
\end{equation}
as $x\to\infty$ for $s\ne t$. As a result, from the fact that the function $\mathcal{H}_{s,t;1}(x)$ on $(-\alpha,\infty)$ is increasing for $0<|t-s|<1$ and decreasing for $|t-s|>1$, it is deduced that the inequality
\begin{equation*}
\bigg[\frac{\Gamma(x+t)}{\Gamma(x+s)}\bigg]^{1/(t-s)} <\frac{(x+s)^{[1/(t-s)+1]/2}}{(x+t)^{[1/(t-s)-1]/2}}
\end{equation*}
on $(-\alpha,\infty)$ holds for $0<|t-s|<1$ and reverses for $|t-s|>1$. Letting $x+s=a$ and $x+t=b$ in the above inequality gives
\begin{equation*}
\bigg[\frac{\Gamma(b)}{\Gamma(a)}\bigg]^{1/(b-a)} <\frac{a^{[1/(b-a)+1]/2}}{b^{[1/(b-a)-1]/2}}
\end{equation*}
which is equivalent to \eqref{gamma-ratio-sqrt-ineq}.
\par
By definition of LCM function and the fact \cite[p.~82]{e-gam-rat-comp-mon} that a completely monotonic function which is non-identically zero cannot vanish at any point on $(0,\infty)$, it is easy to see that when $0<|t-s|<1$, the inequality
\begin{equation}\label{H-theta-ineq}
(-1)^k[\ln\mathcal{H}_{s,t;\lambda}(x)]^{(k)}=(-1)^k[\theta_{s,t;\lambda}(x)]^{(k-1)}>0
\end{equation}
for $k\in\mathbb{N}$ holds if and only if $\lambda\ge\frac1{|t-s|}$ and reverses if and only if $\lambda\ge1$. The inequality~\eqref{H-theta-ineq} may be rewritten as
\begin{multline}
\frac{(-1)^{k-1}\bigl[\psi^{(k-1)}(x+t)-\psi^{(k-1)}(x+s)\bigr]}{t-s}\\* <\frac{(k-1)!}2\biggl[\biggl(\frac1{t-s}+\lambda\biggr)\frac1{(x+s)^k} +\biggl(\lambda-\frac1{t-s}\biggr)\frac1{(x+t)^k}\biggr].
\end{multline}
Consequently, utilizing the complete monotonicity of $\theta_{s,t;\lambda}(x)$ or the logarithmically complete monotonicity of $\mathcal{H}_{s,t;\lambda}(x)$ concludes the double inequality
\begin{multline}\label{final-thm3-square}
\frac{(k-1)!}2\biggl[\biggl(\frac1{t-s}+\beta\biggr)\frac1{(x+s)^k} +\biggl(\beta-\frac1{t-s}\biggr)\frac1{(x+t)^k}\biggr]\\ <\frac{(-1)^{k-1}\bigl[\psi^{(k-1)}(x+t)-\psi^{(k-1)}(x+s)\bigr]}{t-s}\\ <\frac{(k-1)!}2\biggl[\biggl(\frac1{t-s}+\gamma\biggr)\frac1{(x+s)^k} +\biggl(\gamma-\frac1{t-s}\biggr)\frac1{(x+t)^k}\biggr]
\end{multline}
on $(-\alpha,\infty)$ holds either for $0<|t-s|<1$ if and only if $\gamma\ge\frac1{|t-s|}$ and $\beta\le1$ or for $|t-s|>1$ if and only if $\beta\le\frac1{|t-s|}$ and $\gamma\ge1$. Replacing $x+s$ and $x+t$ by $a$ and $b$ respectively in~\eqref{final-thm3-square} leads to~\eqref{final-thm3-square-ab}. The proof of Theorem~\ref{CMDT-divided-thm-3} is complete.
\end{proof}

\section{Remarks}

\begin{rem}
Taking $\lambda=s-t>0$ in Theorem~\ref{CMDT-divided-thm} produces that the function $\frac{\Gamma(x+s)}{\Gamma(x+t)}$ on $(-t,\infty)$ is increasingly convex for $s-t>1$ and increasingly concave for $0<s-t<1$. For detailed information, please refer to \cite[Remark~4.2 and Remark~4.5]{Wendel2Elezovic.tex} on the paper \cite{Kazarinoff-56}.
\end{rem}

\begin{rem}
From the proofs of Theorem~\ref{CMDT-divided-thm} and Theorem~\ref{CMDT-divided-thm-2}, the following conclusions may be summarized: For real numbers $s$, $t$, $\alpha=\min\{s,t\}$ and $\lambda$, the function
\begin{equation}
\theta_{s,t;\lambda}(x)=
\begin{cases}
\dfrac{\psi(x+t)-\psi(x+s)}{t-s} -\dfrac{1+\lambda(2x+s+t)}{2(x+s)(x+t)},&s\ne t\\[0.8em]
\psi'(x+s)-\dfrac{1+2\lambda(x+s)}{2(x+s)^2},&s=t
\end{cases}
\end{equation}
on $(-\alpha,\infty)$ has the following completely monotonic properties:
\begin{enumerate}
\item
For $0<|t-s|<1$,
\begin{enumerate}
\item
the function $\theta_{s,t;\lambda}(x)$ is CM if and only if $\lambda\le1$,
\item
the function $-\theta_{s,t;\lambda}(x)$ is CM if and only if $\lambda\ge\frac1{|t-s|}$;
\end{enumerate}
\item
For $|t-s|>1$,
\begin{enumerate}
\item
the function $\theta_{s,t;\lambda}(x)$ is CM if and only if $\lambda\le\frac1{|t-s|}$,
\item
the function $-\theta_{s,t;\lambda}(x)$ is CM if and only if $\lambda\ge1$;
\end{enumerate}
\item
For $s=t$, the function $\theta_{s,s;\lambda}(x)$ is CM if and only if $\lambda\le1$;
\item
For $|t-s|=1$,
\begin{enumerate}
\item
the function $\theta_{s,t;\lambda}(x)$ is CM if and only if $\lambda<1$,
\item
so is the function $-\theta_{s,t;\lambda}(x)$ if and only if $\lambda>1$,
\item
and $\theta_{s,t;1}(x)\equiv0$.
\end{enumerate}
\end{enumerate}
\end{rem}

\begin{rem}
In \eqref{gamma-ratio-sqrt-ineq}, taking $b=x+1$ and $a=x+\frac12$ yields
\begin{equation}
\biggl[\frac{\Gamma(x+1)}{\Gamma(x+1/2)}\biggr]^2<\biggl(x+\frac12\biggr)\sqrt{\frac{x+1/2}{x+1}}\,
\end{equation}
for $x>-\frac12$, which is a refinement of the inequality
\begin{equation}
\biggl[\frac{\Gamma(x+1)}{\Gamma(x+1/2)}\biggr]^2-x<\frac12
\end{equation}
for $x>-\frac12$, obtained in \cite{waston}.
\end{rem}

\begin{rem}
In~\cite[Lemma~1.2]{batir-new-jipam} and~\cite[Lemma~1.2]{batir-new-rgmia}, it was discovered that if $a\le-\ln2$ and $b\ge0$, then
\begin{equation}\label{batir-ineq-orig}
a-\ln\bigl(e^{1/x}-1\bigr)<\psi(x)<b-\ln\bigl(e^{1/x}-1\bigr)
\end{equation}
holds for $x>0$.
\par
In \cite[Theorem~2.8]{batir-jmaa-06-05-065}, the inequality \eqref{batir-ineq-orig} was sharpened as $a\le-\gamma$ and $b\ge0$.
\par
In \cite{alzer-expo-math-2006}, the function $\phi(x)$ defined by \eqref{phi(x)-dfn} was proved to be strictly increasing on $(0,\infty)$ and
\begin{equation}\label{limit=0}
\lim_{x\to\infty}\phi(x)=0.
\end{equation}
\par
In \cite{property-psi.tex}, among other things, the function $\phi(x)$ was proved to be both strictly increasing and concave on $(0,\infty)$, with $\lim_{x\to0^+}\phi(x)=-\gamma$ and the limit \eqref{limit=0}.
\par
It is not difficult to see that all these results extend, refine and generalize the one-side inequality~\eqref{batir-one-side-ineq} or the increasing property of $\phi(x)$.
\end{rem}

\begin{rem}
After the monotonic and convex properties of the function \eqref{z-s-t-(x)} were perfectly procured in \cite[Theorem~1]{egp}, several alternative proofs were supplied in \cite{201-05-JIPAM, notes-best-simple-open.tex, notes-best-simple.tex, notes-best-new-proof.tex, notes-best.tex-mia, notes-best.tex-rgmia}. The investigation of the function \eqref{z-s-t-(x)} has a long history, see \cite{Kazarinoff-56, Lazarevic, waston} or the survey articles \cite{bounds-two-gammas.tex, Wendel2Elezovic.tex} and related references therein.
\end{rem}

\subsection*{Acknowledgements}
This manuscript was completed during the first author's visit to the RGMIA, Victoria University, Australia, between March 2008 and February 2009. The first author would like to express many thanks to Professors Pietro Cerone and Server S.~Dragomir and other local colleagues for their invitation and hospitality throughout this period.


\begin{thebibliography}{99}

\bibitem{abram}
M. Abramowitz and I. A. Stegun (Eds), \textit{Handbook of Mathematical Functions with Formulas, Graphs, and Mathematical Tables}, National Bureau of Standards, Applied Mathematics Series \textbf{55}, 4th printing, with corrections, Washington, 1965.

\bibitem{alzer-mc-97}
H. Alzer, \textit{On some inequalities for the Gamma and psi functions}, Math. Comp. \textbf{66} (1997), no.~217, 373\nobreakdash--389.

\bibitem{forum-alzer}
H. Alzer, \textit{Sharp inequalities for the digamma and polygamma functions}, Forum Math. \textbf{16} (2004), no.~2, 181\nobreakdash--221.

\bibitem{alzer-expo-math-2006}
H. Alzer, \textit{Sharp inequalities for the harmonic numbers}, Expo. Math. \textbf{24} (2006), no.~4, 385\nobreakdash--388.

\bibitem{alzer-grinshpan}
H. Alzer and A. Z. Grinshpan, \textit{Inequalities for the gamma and $q$-gamma functions}, J. Approx. Theory \textbf{144} (2007), 67\nobreakdash--83.

\bibitem{Atanassov}
R. D. Atanassov and U. V. Tsoukrovski, \textit{Some properties of a class of logarithmically completely monotonic functions}, C. R. Acad. Bulgare Sci. \textbf{41} (1988), no.~2, 21\nobreakdash--23.

\bibitem{batir-interest-jipam}
N. Batir, \textit{An interesting double inequality for Euler's gamma function}, J. Inequal. Pure Appl. Math. \textbf{5} (2004), no.~4, Art.~97; Available online at \url{http://jipam.vu.edu.au/article.php?sid=452}.

\bibitem{batir-interest-rgmia}
N. Batir, \textit{An interesting double inequality for Euler's gamma function}, RGMIA Res. Rep. Coll. \textbf{7} (2004), no.~2, Art.~16; Available online at \url{http://www.staff.vu.edu.au/rgmia/v7n2.asp}.

\bibitem{batir-jmaa-06-05-065}
N. Batir, \textit{On some properties of digamma and polygamma functions}, J. Math. Anal. Appl. \textbf{328} (2007), no.~1, 452\nobreakdash--465; Available online at \url{http://dx.doi.org/10.1016/j.jmaa.2006.05.065}.

\bibitem{batir-new-jipam}
N. Batir, \textit{Some new inequalities for gamma and polygamma functions}, J. Inequal. Pure Appl. Math. \textbf{6} (2005), no.~4, Art.~103; Available online at \url{http://jipam.vu.edu.au/article.php?sid=577}.

\bibitem{batir-new-rgmia}
N. Batir, \textit{Some new inequalities for gamma and polygamma functions}, RGMIA Res. Rep. Coll. \textbf{7} (2004), no.~3, Art.~1, 371\nobreakdash--381; Available online at \url{http://www.staff.vu.edu.au/rgmia/v7n3.asp}.

\bibitem{CBerg}
C. Berg, \textit{Integral representation of some functions related to the gamma function}, Mediterr. J. Math. \textbf{1} (2004), no.~4, 433\nobreakdash--439.

\bibitem{dict-bullen}
P. S. Bullen, \textit{A Dictionary of Inequalities}, Pitman Monographs and Surveys in Pure and Applied Mathematics \textbf{97}, Addison Wesley Longman Limited, 1998.

\bibitem{201-05-JIPAM}
Ch.-P. Chen, \textit{Monotonicity and convexity for the gamma function}, J. Inequal. Pure Appl. Math. \textbf{6} (2005), no.~4, Art.~100; Available online at \url{http://jipam.vu.edu.au/article.php?sid=574}.

\bibitem{chen-qi-log-jmaa}
Ch.-P. Chen and F. Qi, \textit{Logarithmically completely monotonic functions relating to the gamma function}, J. Math. Anal. Appl. \textbf{321} (2006), no.~1, 405\nobreakdash--411.

\bibitem{egp}
N. Elezovi\'c, C. Giordano and J. Pe\v{c}ari\'c, \textit{The best bounds in Gautschi's inequality}, Math. Inequal. Appl. \textbf{3} (2000), 239\nobreakdash--252.

\bibitem{gordon}
L. Gordon, \textit{A stochastic approach to the gamma function}, Amer. Math. Monthly \textbf{101} (1994), no.~9, 858\nobreakdash--865.

\bibitem{grin-ismail}
A. Z. Grinshpan and M. E. H. Ismail, \textit{Completely monotonic functions involving the gamma and $q$\nobreakdash-gamma functions}, Proc. Amer. Math. Soc. \textbf{134} (2006), 1153\nobreakdash--1160.

\bibitem{Ismail-Lorch-Muldoon}
M. E. H. Ismail, L. Lorch, and M. E. Muldoon, \textit{Completely monotonic functions associated with the gamma function and its $q$-analogues}, J. Math. Anal. Appl. \textbf{116} (1986), 1\nobreakdash--9.

\bibitem{Kazarinoff-56}
D. K. Kazarinoff, \textit{On Wallis' formula}, Edinburgh Math. Notes \textbf{1956} (1956), no.~40, 19\nobreakdash--21.

\bibitem{Lazarevic}
I. Lazarevi\'c and A. Lupa\c{s}, \textit{Functional equations for Wallis and Gamma functions}, Publ. Elektrotehn. Fak. Univ. Beograd. Ser. Electron. Telecommun. Automat. No.~\textbf{461-497} (1974), 245\nobreakdash--251.

\bibitem{mpf-1993}
D. S. Mitrinovi\'c, J. E. Pe\v{c}ari\'c and A. M. Fink, \textit{Classical and New Inequalities in Analysis}, Kluwer Academic Publishers, Dordrecht/Boston/London, 1993.

\bibitem{notes-best-simple-open.tex}
F. Qi, \textit{A completely monotonic function involving divided differences of psi and polygamma functions and an application}, RGMIA Res. Rep. Coll. \textbf{9} (2006), no.~4, Art.~8; Available online at \url{http://www.staff.vu.edu.au/rgmia/v9n4.asp}.

\bibitem{bounds-two-gammas.tex}
F. Qi, \textit{Bounds for the ratio of two gamma functions}, RGMIA Res. Rep. Coll. \textbf{11} (2008), no.~3, Art.~1; Available online at \url{http://www.staff.vu.edu.au/rgmia/v11n3.asp}.

\bibitem{Wendel2Elezovic.tex}
F. Qi, \textit{Bounds for the ratio of two gamma functions---From Wendel's limit to Elezovi\'c-Giordano-Pe\v{c}ari\'c's theorem}, Available online at \url{http://arxiv.org/abs/0902.2514}.

\bibitem{notes-best-simple.tex}
F. Qi, \textit{The best bounds in Kershaw's inequality and two completely monotonic functions}, RGMIA Res. Rep. Coll. \textbf{9} (2006), no.~4, Art.~2; Available online at \url{http://www.staff.vu.edu.au/rgmia/v9n4.asp}.

\bibitem{sandor-gamma-2-ITSF.tex}
F. Qi, \textit{Three classes of logarithmically completely monotonic functions involving gamma and psi functions}, Integral Transforms Spec. Funct. \textbf{18} (2007), no.~7, 503\nobreakdash--509.

\bibitem{sandor-gamma-2-ITSF.tex-rgmia}
F. Qi, \textit{Three classes of logarithmically completely monotonic functions involving gamma and psi functions}, RGMIA Res. Rep. Coll. \textbf{9} (2006), Suppl., Art.~6; Available online at \url{http://www.staff.vu.edu.au/rgmia/v9(E).asp}.

\bibitem{notes-best-simple-open.tex-rev}
F. Qi and B.-N. Guo, \textit{A class of completely monotonic functions involving divided differences of the psi and polygamma functions and some applications}, Available online at \url{http://arxiv.org/abs/0903.1430}.

\bibitem{notes-best-new-proof.tex}
F. Qi and B.-N. Guo, \textit{An alternative proof of Elezovi\'c-Giordano-Pe\v{c}ari\'c's theorem}, Available online at \url{http://arxiv.org/abs/0903.1174}.

\bibitem{minus-one}
F. Qi and B.-N. Guo, \textit{Complete monotonicities of functions involving the gamma and digamma functions}, RGMIA Res. Rep. Coll. \textbf{7} (2004), no.~1, Art.~8, 63\nobreakdash--72; Available online at \url{http://www.staff.vu.edu.au/rgmia/v7n1.asp}.

\bibitem{AAM-Qi-09-PolyGamma.tex}
F. Qi and B.-N. Guo, \textit{Necessary and sufficient conditions for functions involving the tri- and tetra-gamma functions to be completely monotonic}, Adv. Appl. Math. (2009), in press.

\bibitem{Infinite-family-Digamma.tex}
F. Qi and B.-N. Guo, \textit{Sharp inequalities for the psi function and harmonic numbers}, Available online at \url{http://arxiv.org/abs/0902.2524}.

\bibitem{property-psi-ii.tex}
F. Qi and B.-N. Guo, \textit{A simple proof of monotonicity of a function involving the psi and exponential functions}, submitted.

\bibitem{property-psi.tex}
F. Qi and B.-N. Guo, \textit{Some properties of the psi and polygamma functions}, Available online at \url{http://arxiv.org/abs/0903.1003}.

\bibitem{notes-best-simple-rev.tex}
F. Qi and B.-N. Guo, \textit{Completely monotonic functions involving divided differences of the di- and tri-gamma functions and some applications}, Commun. Pure Appl. Anal. (2010), in press.

\bibitem{e-gam-rat-comp-mon}
F. Qi, B.-N. Guo and Ch.-P. Chen, \textit{Some completely monotonic functions involving the gamma and polygamma functions}, J. Aust. Math. Soc. \textbf{80} (2006), 81\nobreakdash--88.

\bibitem{auscm-rgmia}
F. Qi, B.-N. Guo and Ch.-P. Chen, \textit{Some completely monotonic functions involving the gamma and polygamma functions}, RGMIA Res. Rep. Coll. \textbf{7} (2004), no.~1, Art.~5, 31\nobreakdash--36; Available online at \url{http://www.staff.vu.edu.au/rgmia/v7n1.asp}.

\bibitem{notes-best.tex-mia}
F. Qi, B.-N. Guo and Ch.-P. Chen, \textit{The best bounds in Gautschi-Kershaw inequalities}, Math. Inequal. Appl. \textbf{9} (2006), no.~3, 427\nobreakdash--436.

\bibitem{notes-best.tex-rgmia}
F. Qi, B.-N. Guo and Ch.-P. Chen, \textit{The best bounds in Gautschi-Kershaw inequalities}, RGMIA Res. Rep. Coll. \textbf{8} (2005), no.~2, Art.~17; Available online at \url{http://www.staff.vu.edu.au/rgmia/v8n2.asp}.

\bibitem{waston}
G. N. Watson, \textit{A note on gamma functions}, Proc. Edinburgh Math. Soc. \textbf{11} (1958/1959), no.~2, Edinburgh Math Notes No.~42 (misprinted 41) (1959), 7\nobreakdash--9.

\bibitem{wendel}
J. G. Wendel, \textit{Note on the gamma function}, Amer. Math. Monthly \textbf{55} (1948), no.~9, 563\nobreakdash--564.

\bibitem{widder}
D. V. Widder, \textit{The Laplace Transform}, Princeton University Press, Princeton, 1946.

\end{thebibliography}
\end{document}